\documentclass[12pt]{amsart}
\usepackage{amsfonts,latexsym,amsthm,amssymb,graphicx}
\usepackage[all]{xy}
\usepackage{hyperref}
\DeclareFontFamily{OT1}{rsfs}{}
\DeclareFontShape{OT1}{rsfs}{n}{it}{<-> rsfs10}{}
\DeclareMathAlphabet{\mathscr}{OT1}{rsfs}{n}{it}

\CompileMatrices

\newtheorem{theorem}{Theorem}[section]
\newtheorem{lemma}[theorem]{Lemma}
\newtheorem{corol}[theorem]{Corollary}
\newtheorem{prop}[theorem]{Proposition}
\newtheorem{claim}[theorem]{Claim}

{\theoremstyle{definition} \newtheorem{defin}[theorem]{Definition}}
{\theoremstyle{remark} \newtheorem{remark}[theorem]{Remark}
\newtheorem{example}[theorem]{Example}}

\newenvironment{demo}{\textbf{Proof.}}{\qed}

\newcommand{\Abb}{{\mathbb{A}}}
\newcommand{\Cbb}{{\mathbb{C}}}
\newcommand{\Pbb}{{\mathbb{P}}}
\newcommand{\Qbb}{{\mathbb{Q}}}
\newcommand{\Tbb}{{\mathbb{T}}}
\newcommand{\Zbb}{{\mathbb{Z}}}

\newcommand{\cA}{{\mathscr A}}
\newcommand{\cC}{{\mathscr C}}
\newcommand{\cE}{{\mathscr E}}
\newcommand{\cI}{{\mathscr I}}
\newcommand{\cL}{{\mathscr L}}
\newcommand{\cO}{{\mathscr O}}

\newcommand{\one}{1\hskip-3.5pt1}
\newcommand{\csm}{{c_{\text{SM}}}}
\newcommand{\hcA}{{\hat \cA}}

\DeclareMathOperator{\rk}{rk}

\DeclareMathOperator{\Var}{Var}
\DeclareMathOperator{\Der}{Der}

\newcommand{\qede}{\hfill$\lrcorner$}

\title[Euler characteristics and Chern classes]{
Euler characteristics of general linear sections 
and polynomial Chern classes
}
\author{Paolo Aluffi}
\address{
Mathematics Department, 
Florida State University,
Tallahassee FL 32306, U.S.A.
}
\email{aluffi@math.fsu.edu}

\begin{document}

\begin{abstract}
We obtain a precise relation between the Chern-Schwartz-MacPherson class
of a subvariety of projective space and the Euler characteristics of its general
linear sections. In the case of a hypersurface, this leads to simple proofs of
formulas of Dimca-Papadima and Huh for the degrees of the polar map of a
homogeneous polynomial, extending these formula to any algebraically closed
field of characteristic~$0$, and proving a conjecture of Dolgachev on 
`homaloidal' polynomials in the same context. We generalize these
formulas to subschemes of higher codimension in projective space.

We also describe a simple approach to a theory of `polynomial Chern classes'
for varieties endowed with a morphism to projective space, recovering 
properties analogous to the Deligne-Grothendieck axioms from basic properties
of the Euler characteristic. We prove that the polynomial Chern class defines
homomorphisms from suitable relative Grothendieck rings of varieties to
$\Zbb[t]$.
\end{abstract}

\maketitle


\section{Introduction}\label{intro}

\subsection{}\label{introsu}
Let $X$ be a projective variety over an algebraically closed field $k$ of 
characteristic~$0$, endowed with a specific embedding in projective space. 
If $k=\Cbb$, one of the most important invariants of $X$ is its 
topological Euler characteristic, $\chi(X)$. There is a natural generalization of this 
invariant to arbitrary algebraically closed fields of characteristic zero:
if $X$ is nonsingular, we may take $\chi(X)$ to equal the degree of $c(TX)\cap [X]$;
and the singular case may be dealt with by using resolution of singularities 
(see~\S\ref{thmsec} for details). The resulting invariant has the expected properties
of the topological Euler characteristic: it is multiplicative on products, it satisfies 
inclusion-exclusion, and in particular it can be consistently defined for any 
locally closed subset of $\Pbb^n$.

While $\chi(X)$ does not depend on the embedding of $X$ into $\Pbb^n$, we can
access more refined invariants of the embedding by considering general linear
sections. We let $X_r=X\cap H_1\cap\cdots\cap H_r$ be the intersection of $X$
with $r$ general hyperplanes, and we assemble the Euler characteristics of these
loci into a generating polynomial of degree~$\le n$:
\[
\chi_X(t):=\sum_{r\ge 0} \chi(X_r)\cdot (-t)^r\quad.
\]
This polynomial may be defined for any locally closed subset of $\Pbb^n$.
Our main theme in this note is an interpretation of $\chi_X(t)$ in terms of the
{\em Chern-Schwartz-MacPherson\/} class $\csm(X)$. This is a class in
the Chow group of $X$, generalizing to (possibly) singular varieties the 
total Chern class of the tangent bundle of $X$ in the nonsingular case, and
satisfying a compelling functoriality property, which will be recalled 
in~\S\ref{thmsec}. Chern-Schwartz-MacPherson classes can also be defined 
over any algebraically closed field of characteristic~$0$, and the 
functoriality property mentioned above implies that if $X$ is complete, then the 
degree $\int \csm(X)$ equals $\chi(X)$.
In fact, a $\csm$ class may be defined for any constructible function on a
variety, and here we will associate with $X\subseteq \Pbb^n$ the class
$\csm(\one_X)\in A_*\Pbb^n$. (As a template to keep in mind, the information 
carried by this class for a nonsingular projective $X\subseteq \Pbb^n$ amounts
to the degrees of the components of different dimension in $c(TX)\cap [X]$.)
We will also write this class as a polynomial of degree $\le n$: we let
\[
\gamma_X(t)=\sum_{r\ge 0} \gamma_r\, t^r
\]
be the polynomial obtained from $\csm(\one_X)$ by replacing $[\Pbb^r]$
with $t^r$. This is another polynomial defined for any locally closed 
$X\subseteq \Pbb^n$. (In our template, $\gamma_r$ is the degree of 
$c_{\dim X-r}(TX)\cap [X]$ as a class in $\Pbb^n$.) 

We will prove that for all locally closed subsets of $\Pbb^n$ the polynomials 
$\gamma_X(t)$, $\chi_X(t)$ carry
precisely the same information. We consider the following linear transformation:
\[
p(t) \mapsto \cI(p):=\frac{t\cdot p(-t-1)+p(0)}{t+1}\quad.
\]
It is clear that if $p$ is a polynomial, then $\cI(p)$ is a polynomial of the same
degree. Also, it is immediately checked that $\cI(p)(0)=p(0)$, and 
$\cI(p)(-1)=p(0)+p'(0)$.
Further, $\cI$ is an {\em involution;\/}
in fact, if $p(t)=p(0)+t p_+(t)$, then $\cI(p)=p(0)-t p_+(-t-1)$, so that the effect
of $\cI$ is to perform a sign-reversing symmetry about $t=-1/2$ of the 
non-constant part of $p$.

\begin{theorem}\label{main}
For every locally closed set $X\subseteq \Pbb^n$, the involution $\cI$ interchanges 
$\chi_X(t)$ and $\gamma_X(t)$:
\[
\gamma_X=\cI(\chi_X)\quad,\quad \chi_X=\cI(\gamma_X)\quad.
\]
\end{theorem}

Theorem~\ref{main} is a straightforward exercise for $X$ nonsingular
and projective. Its extension to arbitrarily singular quasi-projective varieties is 
not technically demanding, but appears to carry significant information.

\subsection{}\label{introDPH}
To see why Theorem~\ref{main} may be interesting, consider the (very) special
case
in which~$X$ is the complement $D(F)$ of a hypersurface in $\Pbb^n$, given 
by the vanishing of a homogeneous polynomial $F(x_0,\dots,x_n)$. Using an 
expression for the $\csm$ class from \cite{MR1956868}, it is easy to show that 
the degree of the `polar' (or `gradient') map $\Pbb^n \dashrightarrow \Pbb^n$ 
given by $(\partial F/\partial x_0,\dots,\partial F/\partial x_n)$ equals 
$(-1)^n\gamma_{D(F)}(-1)$ (see~\S\ref{polardegs}). By Theorem~\ref{main},
\[
\gamma_{D(F)}(-1)=\chi_{D(F)}(0)+\chi'_{D(F)}(0)=\chi(D(F))
-\chi(D(F)\cap H)=\chi(D(F)\smallsetminus H)
\]
for a general hyperplane $H$. Over $\Cbb$, this formula for the polar degree was
obtained by Dimca and Papadima as a consequence of their description of the
homotopy type of the complement $D(F)$ (\cite{MR2018927}, Theorem~1). 
The argument deriving this formula from Theorem~\ref{main} works over any 
algebraically closed field of characteristic zero, and hence it also extends to this 
context the consequence that the degree of the polar map only depends on the 
{\em reduced\/} polynomial $F_\text{red}$. In particular, $F$ is `homaloidal'
if and only if $F_\text{red}$ is; this fact was conjectured by Dolgachev 
(\cite{MR1786486}, \S3).
More generally, the argument extends easily to yield a formula for the 
{\em Huh-Teissier-Milnor\/} numbers~$\mu^{(i)}$ defined in \cite{Huh}
in terms of mixed multiplicities. Theorem~\ref{main} implies
\[
\mu^{(i)}=(-1)^i \chi(D(F)\cap (\Pbb^i\smallsetminus \Pbb^{i-1}))\quad,
\]
where $\Pbb^i$ denotes a general linear subspace of dimension~$i$
(Corollary~\ref{Huhf}).
This formula is given in Theorem~9 (1) of \cite{Huh}, as a consequence
of an explicit topological description of the intersection $D(f)\cap \Pbb^i$.
The argument sketched above does 
less, since it only yields the numerical consequence of this more
refined topological information; but it proves the validity of this formula
over any algebraically closed field of characteristic zero, and is in a
sense more straightforward.

This approach also allows us to generalize some of these considerations to 
arbitrary codimension. We propose a definition of `polar degrees' for arbitrary
subschemes $X\subseteq \Pbb^n$, giving a relation between these degrees
and the Euler characteristics of sections of $\Pbb^n\smallsetminus X$
generalizing the Dimca-Papadima-Huh formulas for hypersurfaces recalled 
above. We prove
that these polar degrees only depend on the support $X_\text{red}$,
thus generalizing Dolgachev's conjecture to arbitrary subschemes of $\Pbb^n$.
The polar degrees of $X$ are given in terms of generators for an ideal defining
$X$; the fact that they are independent of the choice of the generators
would deserve to be understood from a more explicitly algebraic
viewpoint. In our approach, the independence follows from the relation 
with Chern-Schwartz-MacPherson classes. 

Providing algebro-geometric proofs of the formula for the polar degree and
Dolgachev's conjecture was a natural problem. Fassarella, Pereira, and 
Medeiros have developed an algebro-geometric approach through foliations 
(\cite{MR2361094}, \cite{FM}), and also obtain the Dimca-Papadima formula
and Huh's generalization to higher order polar degrees; their results are
stated for complex hypersurfaces.
We note that some of the beautiful formulas for the polar degrees obtained
in~\cite{FM} have a natural explanation when the degrees are viewed in
terms of $\csm$ classes. (For example, formula (3) in \cite{FM} is an 
expression of the inclusion-exclusion property satisfied by $\csm$ classes.)

Previous work on Dolgachev's conjecture and homaloidal polynomials also
includes \cite{MR1879808}; \cite{MR2369940} (for product of linear forms, 
over fields of arbitrary characteristic); \cite{MR2431661}; and \cite{MR2654229}.
We are not aware of work on polar degrees in higher codimension.
Over $\Cbb$, the relation between $\csm$ classes and the Huh-Teissier-Milnor 
numbers is also observed in \cite{Huh} and further employed very effectively 
in the recent preprint~\cite{Huh2}, which also includes applications to the problem
of studying homaloidal polynomials.

\subsection{}
Details for the application to Dimca-Papadina/Huh formulae, and the 
generalization to higher codimension, are given in
\S\ref{DPH}. In \S\ref{EC} we sketch a general framework suggested by
Theorem~\ref{main}. The fact that the information carried by the naive Euler 
polynomial $\chi_X(t)$ and the more sophisticated Chern class polynomial
$\gamma_X(t)$ is precisely the same indicates that it should be possible
to give a simple treatment of `polynomial' $\csm$ classes, based solely on 
the Euler
characteristic. The target of the $\csm$ natural transformation is the Chow
functor; while this is a virtue of the full theory, it is an obstacle if one is
interested in e.g., studying the motivic nature of these classes. For example, 
while $\csm$ classes satisfy a scissor relation, they do not factor through the 
Grothendieck group of varieties, simply because their target depends on the
variety. We propose a theory of Chern classes with constant {\em polynomial\/}
target for varieties endowed with a map to projective space. This theory
can be defined solely in terms of Euler characteristics and the involution 
appearing in Theorem~\ref{main}, and its main covariance property, which
parallels closely the functoriality of $\csm$ classes, is a simple cut-and-paste
exercise (cf.~Lemmata~\ref{fiberlemma} and~\ref{covarlem}). It is immediate
that the resulting `polynomial Chern classes' $c_*$ with values in $\Zbb[t]$ 
factor through the Grothendieck
group of varieties endowed with morphisms to projective space. 
We prove that the classes $c_*$ also define {\em ring\/} homomorphisms, 
preserving different products one can define on this Grothendieck group
(Propositions~\ref{mult} and~\ref{mult2}). This is useful in concrete computations.

Relative Grothendieck groups of varieties have been applied to a very general
theory of characteristic classes for singular varieties in~\cite{MR2646988}.

{\bf Acknowledgements.} The author thanks Ettore Aldrovandi, Xia Liao, Miguel
Marco-Buzun\'ariz, and Matilde 
Marcolli for useful remarks, and Caltech for hospitality during the preparation of this 
note. The fact that the Chern-Schwartz-MacPherson class of an embedded variety
carries information on Euler characteristic of general linear sections was pointed
out by Deligne in a reply to a letter the author sent him at the time of the preparation
of~\cite{MR2282409}.


\section{Euler and Chern}\label{thmsec}

\subsection{The Euler characteristic}\label{introEc}
Throughout the paper, we work over a fixed algebraically closed field $k$ of 
characteristic~$0$. Part of our goal is to emphasize that familiar notions from
complex geometry generalize to this context, so for once we do not encourage
the reader to assume that $k=\Cbb$. Morphisms are implicitly assumed to be 
separable, of finite type; `point' will mean `closed point'.

The {\em Grothendieck group\/} of $k$-varieties, $K(\Var_k)$ is the abelian 
group generated by isomorphism classes of $k$-varieties modulo the scissor
relation
\[
[X] = [U] + [Z]
\]
for every closed $Z\subseteq X$, $U=X\smallsetminus Z$. The operation defined
on generators by $[X]\cdot [Y]:=[X\times Y]$ endows $K(\Var_k)$ with a structure
of ring.

\begin{lemma}\label{Bittner}
There is a unique ring homomorphism $\chi:K(\Var_k) \to \Zbb$ such that if $X$ is
nonsingular and projective, then $\chi([X])=\int c(TX)\cap [X]$.
\end{lemma}

\begin{proof}
By Theorem~3.1 in~\cite{MR2059227}, $K(\Var_k)$ admits an alternative
presentation as the abelian group generated by isomorphism classes of smooth
projective $k$-varieties modulo the relation $[B\ell_ZX]-[E]=[X]-[Z]$ for $X$
smooth and projective and $Z\subseteq X$ a closed smooth subvariety; here
$B\ell_ZX$ is the blow-up of $X$ along $Z$ and $E$ is the exceptional divisor.
It suffices therefore to verify that the degree of the top Chern class of the
tangent bundle satisfies these relations, and the relation defining the ring 
structure. This latter check is immediate. As for the blow-up relations,
since $p:E\to Z$ may be identified with the projective normal bundle 
of $Z$ in $X$, $c(TE)$ is determined by the Euler sequence:
\[
\xymatrix{
0 \ar[r] &
\cO \ar[r] &
p^*N_ZX\otimes \cO(1) \ar[r] &
TE \ar[r] & 
p^*TZ \ar[r] &
0\quad.
}
\]
From this it is straightforward to verify that
\[
\int c(TE)\cap [E] = (d+1)\int c(TZ)\cap [Z]\quad,
\]
where $d=\rk N_ZX$ is the codimension of $Z$ in $X$. Thus, what needs to
be checked is that
\[
\int c(TB\ell_ZX)\cap [B\ell_Z X] - \int c(TX)\cap [X] = d \int c(TZ)\cap [Z]\quad,
\]
for $Z$ a closed smooth subvariety of a smooth projective variety $X$.
This can be done by using the explicit formula for blowing up Chern classes
given in~\cite{85k:14004}, Theorem~15.4.
\end{proof}

We will write $\chi(U)$ for the value taken by the homomorphism $\chi$ on
$[U]$, and call this number the {\em Euler characteristic\/} of $U$. 
Every locally closed subset of a complete variety has a class in $K(\Var_k)$,
so every such set has a well-defined Euler characteristic. Of course if 
$k=\Cbb$, then $\chi$ agrees with the ordinary topological Euler characteristic 
(with compact support).

\begin{remark}
The fact that $\chi$ defines a homomorphism $K(\Var_k)\to \Zbb$ captures
the usual properties of the ordinary Euler characteristic: inclusion-exclusion
($\chi(X\cup Y)=\chi(X)+\chi(Y)-\chi(X\cap Y)$) and multiplicativity on locally 
trivial fibrations. Also, if $X \to Y$ is smooth and proper, then $\chi(X)
=\chi(Y)\cdot \chi_f$, where $\chi_f$ is the Euler characteristic of any fiber.
\end{remark}

This seems the appropriate place to point out the following observation,
that will be used in~\S\ref{EC}:

\begin{lemma}\label{fiberlemma}
Let $f: X\to Y$ be a morphism of $k$-varieties. Then there exist subvarieties 
$V_1,\dots,V_r$ of $Y$ and integers $m_1,\dots, m_r$ such that $\forall p\in Y$,
\[
\chi(f^{-1}(p))=\sum_{V_j\ni p} m_j\quad.
\]
Further, for every subvariety $W\subseteq Y$,
\[
\chi(f^{-1}(W))=\sum_{i=1}^r m_j\, \chi(W\cap V_j)\quad.
\]
\end{lemma}

Indeed, the existence of $V_1,\dots,V_r$ follows from the fact that the Euler 
characteristic of fibers is constant on a nonempty open set.
Both this fact and the second assertion follow from standard techniques: Nagata's
embedding theorem, resolution of singularities, generic smoothness, and 
inclusion-exclusion for $\chi$ may be used to reduce to the case of $f$ smooth and 
proper, for which the assertions are trivial. Details are left to the reader, and may 
be distilled from~\cite{MR2282409}, \S5.4-6.
\qede

\subsection{Chern-Schwartz-MacPherson classes}\label{introCSM}
For a variety $X$, we denote by $\cC(X)$ the abelian group of $\Zbb$-valued 
{\em constructible functions\/} on $X$; thus, every $\varphi\in \cC(X)$ may be
written as a finite sum $\sum_i n_i \one_{Z_i}$ where $n_i\in \Zbb$, $Z_i$ are
subvarieties of $X$, and $\one_{Z_i}$ is the function giving $1$ for $p\in Z_i$
and $0$ for $p\not\in Z_i$. The assignment $X\mapsto \cC(X)$ defines a
covariant functor to abelian groups: if $f: X\to Y$ is a morphism, we may define 
a push-forward
\[
f_*: \cC(X) \to \cC(Y)
\]
by letting $f_*(\one_Z)(p) = \chi(Z\cap f^{-1}(p))$ for any subvariety $Z\subseteq X$
and $p\in Y$, and extending by linearity. Note that according to this definition
\[
f_*(\one_X) = \sum_i m_i \one_{V_i}\quad,
\]
where the varieties $V_i$ are those appearing in Lemma~\ref{fiberlemma}.

On the category of complete $k$-varieties and proper morphisms there exists a 
unique natural transformation $\cC\leadsto A_*$, normalized by the condition that if
$X$ is nonsingular and complete, then $\one_X \mapsto c(TX)\cap [X]$.
Over $\Cbb$ and in homology, this fact is due to R.~MacPherson 
(\cite{MR0361141}). With suitable positions, the class associated with $\one_X$
for a (possibly) singular $X$ agrees with the class previously defined by 
M.-H.~Schwartz (\cite{MR35:3707}, \cite{MR32:1727}). The theory was 
extended to arbitrary algebraically closed fields of characteristic~$0$ in
\cite{MR1063344}; the treatment in \cite{MR2282409} includes an extension
to non-complete varieties and not necessarily proper morphisms.
We call the class associated with a constructible function $\varphi$ on a
variety $X$ the `Chern-Schwartz-MacPherson class' of $\varphi$,
denoted $\csm(\varphi)$. If $Z\subseteq X$ and the context is clear,
we denote by $\csm(Z)$ the class $\csm(\one_Z)\in A_*X$.

The normalization and functoriality properties of $\csm$:
\begin{itemize}
\item $\csm(\one_X) = c(TX)\cap [X]$ for $X$ nonsingular and complete, and
\item $f_* \csm(\varphi) = \csm(f_*\varphi)$ for $f: X\to Y$ a proper morphism
\end{itemize}
(and linearity) determine $\csm$ uniquely, by resolution of singularities.
We note that if $X_1,X_2\subseteq X$, then $\one_{X_1\cup X_2}
=\one_{X_1}+\one_{X_2}-\one_{X_1\cap X_2}$, and hence
\[
\csm(X_1\cup X_2)=\csm(X_1)+\csm(X_2)-\csm(X_1\cap X_2)
\quad\text{in $A_*X$:}
\]
thus, $\csm$ classes (like $\chi$) satisfy inclusion-exclusion. In fact, the
degree of $\csm(X)$ agrees with $\chi(X)$: if $X$ is complete, then
\begin{equation}\label{degcsm}
\int\csm(X)=\chi(X)\quad.
\end{equation}
To see this, apply functoriality to the constant map from $X$ to a point.
This simple observation will be crucial in what follows.

\subsection{The $\csm$ class of a hypersurface}\label{csmhyp}
As this will be needed in a proof given below, we recall an expression for 
$\csm(X)$ in the case that $X$ is a 
hypersurface in a nonsingular variety $V$. We will use the following notation: 
if $a\in A_{\dim V-i}V$ is a class in codimension~$i$, and $\cL$ is a line bundle 
on $V$, we let
\[
a^\vee:=(-1)^i a\quad,\quad a\otimes \cL = \frac{a}{c(\cL)^i}\quad,
\]
and extend these operations to $A_*V$ by linearity.
(As $c(\cL)=1+ c_1(\cL)$, and $c_1(\cL)$ is nilpotent, $c(\cL)$ has an 
inverse as an operator over $A_*V$; this is what the notation $a/c(\cL)$ means.) 
For properties satisfied by these
operations, we address the reader to~\cite{MR96d:14004}, \S2. In particular, 
$(A\otimes \cL_1)\otimes \cL_2=A\otimes (\cL_1\otimes \cL_2)$
for all line bundles $\cL_1$, $\cL_2$, and
\[
(c(\cE)\cap A)\otimes \cL =c(\cE\otimes \cL)\cdot c(\cL)^{-\rk \cE} \cap (A\otimes \cL)
\]
for all $A\in A_*V$ and all vector bundles $\cE$ on $V$. We routinely abuse
language and write $A\otimes \cL$ for $A\in A_*Y$ if $Y\subseteq V$; if we 
need to emphasize that the codimension is taken in $V$, we write $A\otimes_V \cL$.

\begin{theorem}[\cite{MR2001i:14009}, Theorem~I.4]\label{1995thm}
Let $X$ be a hypersurface in a nonsingular complete variety $V$. Then
\[
\csm(\one_X)=c(TV)\cap \left(\frac{[X]}{1+X}+\frac 1{1+X} \cap \left( s(JX,V)^\vee
\otimes_V \cO(X)\right)\right)
\]
in $A_*V$.
\end{theorem}

Here $JX$ denotes the {\em singularity subscheme\/} of $X$, locally defined
in $V$ by a local generator $F$ for the ideal of $X$ and by $\partial F$
as $\partial$ ranges over local sections of $\Der_V$. (Informally, $JX$
is defined by $F$ and its partial derivatives.) We use the Segre class
$s(-,-)$ in the sense of \cite{85k:14004}, and omit evident pull-backs and
push-forwards.

\begin{remark}\label{csmred}
The right-hand side of the formula in Theorem~\ref{1995thm} makes sense
for hypersurfaces with (possibly) multiple components: multiple components 
of $X$ appear as components of $JX$. It is a remarkable feature of this 
expression that it does {\em not\/} depend on the multiplicities of the components:
the change in $s(JX,V)$ due to the presence of multiplicities is precisely 
compensated by the other ingredients in the expression.
This is observed in \cite{MR2001i:14009}, \S2.1; briefly, the blow-up formula
proved in~\S3 of \cite{MR2001i:14009} reduces this fact to the simple normal
crossing case, where it can be worked out explicitly.

This is compatible with the fact that the left-hand side, 
$\csm(\one_X)$, should ignore the multiplicities of the components of $X$
because $\one_X$ is determined by set-theoretic information: if $U$ is the 
complement of $X$ in $V$, $\one_X=\one_V-\one_U=\one_{X_\text{red}}$. 
\qede\end{remark}

\subsection{$\csm$ classes and general hyperplane sections}
We now assume that $V=\Pbb^n$, and consider general hyperplane sections
of $\csm$ classes of locally closed subsets.

\begin{prop}\label{csmhypsec}
Let $U\subseteq \Pbb^n$ be any locally closed set (so that $\one_U$ is a constructible
function). Then for a general hyperplane $H\subseteq \Pbb^n$,
\[
\csm(\one_{U\cap H}) = \frac{H}{1+H} \cdot \csm(\one_U)
\]
in $A_*V$.
\end{prop}

\begin{demo}
Since $U$ may be written as a set-difference of two closed sets, we may assume
that $U$ is itself closed, by additivity of $\csm$ classes. Since every closed subset
may be written as an intersection of hypersurfaces, by inclusion-exclusion we
may assume that $U=X$ is a hypersurface of $\Pbb^n$. Let $H\cong\Pbb^{n-1}$ be 
a general hyperplane and $X'=X\cap H$. By Theorem~\ref{1995thm},
\begin{align*}
\csm(\one_X) & = c(T\Pbb^n)\cap \left(\frac{[X]}{1+X}+\frac 1{1+X} \cap \left( 
s(JX,\Pbb^n)^\vee\otimes_{\Pbb^n} \cO(X)\right)\right)\quad\text{in $A_*\Pbb^n$, and} \\
\csm(\one_{X'}) & = c(T\Pbb^{n-1})\cap \left(\frac{[X']}{1+X'}+\frac 1{1+X'} \cap \left( 
s(JX',H)^\vee\otimes_{H} \cO(X')\right)\right)\quad\text{in $A_*\Pbb^{n-1}$.}
\end{align*}
It is a good exercise in the notation introduced in~\S\ref{csmhyp} to verify that the 
equality of these two classes is equivalent to
\begin{equation}\label{segreclassrel}
H\cdot s(JX,\Pbb^n) = s(JX',H)
\end{equation}
(cf.~\cite{ccdc}, \S3.2). Thus, we are reduced to proving~\eqref{segreclassrel}.
This follows from two observations:
\begin{itemize}
\item $H\cdot s(JX,\Pbb^n)= s(H\cap JX,H)$ if $H$ intersects properly the supports of 
the cone of $JX$ in $\Pbb^n$; and
\item For a general hyperplane $H$, $s(JX',H) = s(H\cap JX,H)$.
\end{itemize}
The first assertion may be verified by comparing the blow-up of $\Pbb^n$
along $JX$ and the blow-up of $H$ along $H\cap JX$; details may be found in
e.g., the proof of Claim~3.2 in~\cite{ccdc}. For the second assertion, the ideals
of $JX'$ and $H\cap JX$ have the same integral closure by Teissier's idealistic
Bertini, \cite{MR58:27964}, \S2.8 (see~\cite{Huh}, Lemma~31 for a transparent 
proof in the homogeneous case that does not depend on complex geometry). 
Subschemes with the same integral closure have the same Segre class since
they have the same normalized blow-up, and Segre classes are birational
invariants (\cite{85k:14004}, Proposition~4.2).

This verifies \eqref{segreclassrel}, concluding the proof of the proposition.
\end{demo}

\subsection{Proof of Theorem~\ref{main}}\label{proofmt}
Let $X$ be a locally closed subset of $\Pbb^n$. As in \S\ref{intro}, we let
\[
\chi_X(t):= \sum_{r\ge 0} (-1)^r \chi(H_1\cap\cdots \cap H_r\cap X)\, t^i\quad,
\]
where $H_1,H_2,\dots$ are general hyperplanes. Also, we let
\[
\gamma_X(t):= \sum_{r\ge 0} \left(\int H^r\cdot \csm(\one_X)\right) t^r
\]
where $H$ is the hyperplane class. That is, $\gamma_X(t)$ is obtained from
$\csm(\one_X)=\sum_{r\ge 0} c_r [\Pbb^r]$ by replacing $[\Pbb^r]$ by $t^r$.
Iterating Proposition~\ref{csmhypsec}, we obtain 
that with notation as above,
\[
\csm(\one_{H_1\cap\cdots\cap H_r\cap X}) = \frac{H^r}{(1+H)^r}\cap \csm(\one_X)
\]
in $A_*\Pbb^n$; by \eqref{degcsm},
\[
\chi(H_1\cap\cdots \cap H_r\cap X) = \int \frac{H^r}{(1+H)^r}\cap 
\left(\sum_{\ell\ge 0} c_\ell [\Pbb^\ell]\right)
\]
and hence
\begin{align*}
\chi(t)&=\sum_{r\ge 0} \left(\int \frac{(-H)^r}{(1+H)^r}\cap \csm(\one_X)\right)t^r 
=\sum_{r\ge 0} \left(\int \frac{(-H)^r}{(1+H)^r}\cap \sum_{\ell\ge 0} c_\ell [\Pbb^\ell]\right)t^r \\
&=\sum_{\ell\ge 0} c_\ell \sum_{r\ge 0} \sum_{k\ge 0} \binom{r+k-1}k (-H)^{r+k}\cdot [\Pbb^\ell] t^r \\
&=c_0+t\sum_{\ell\ge 1} c_\ell \sum_{k\ge 0} \binom{\ell-1}k (-1)^\ell t^{\ell-1-k} \\
&=c_0-t\sum_{\ell\ge 1} c_\ell \,(-1)^{\ell-1}\sum_{k\ge 0} \binom{\ell-1}k t^{\ell-1-k} \\
&=c_0-t\sum_{\ell\ge 1} c_\ell \,(-t-1)^{\ell-1} \\
&=\frac{\gamma_X(0) + t\cdot \gamma_X(-t-1)}{t+1}\quad.
\end{align*}
With notation as in \S\ref{intro}, this verifies that $\chi_X=\cI(\gamma_X)$ and 
concludes the proof of Theorem~\ref{main}.

\subsection{}
The fact that the polynomial $\gamma_X(t)$ may be recovered from the collection
of Euler characteristics of general hyperplane sections follows directly from the
good behavior of $\csm$ under general hyperplane sections (verified in
Proposition~\ref{csmhypsec}). If a class has the same behavior, and its degree
equals the Euler characteristic, then the class must agree with the $\csm$ class.
This strategy was used in~\cite{MR96d:14004} to prove a numerical version of
Theorem~\ref{1995thm}; it has also been used recently by Liao in studying the 
relation between the Chern class of the bundle of logarithmic derivations of a 
free divisors and the $\csm$ class of the complement of the divisor (\cite{Liao}).

Other classes have the same behavior under general hyperplane sections. 
For example, for a hypersurface $X$ of a 
nonsingular variety $V$, let $\pi(X)$ denote the {\em Parusi\'nski-Milnor
number\/} of $X$: this is an integer defined for arbitrary hypersurfaces, and
agreeing with the sum of the Milnor numbers at the singularities if these are
all isolated. This invariant is defined and studied in~\cite{MR949831} (over $\Cbb$).
For a hypersurface $X$ of~$\Pbb^n$, we can define the polynomial
\[
\pi_X(t):=\sum_{r\ge 0} (-1)^r \pi(H_1\cap\cdots \cap H_r\cap X)\, t^i\quad,
\]
where $H_1,H_2,\dots$ are general hyperplanes. On the other hand, we can
consider the {\em Milnor class\/} of $X$, defined in~\cite{MR2002g:14005}; 
its push-forward to $\Pbb^n$ is a class $\nu_0 [\Pbb^0] + \nu_1 [\Pbb^1] + \cdots$,
with which we associate the polynomial
\[
\nu_X(t):= \sum_{r\ge 0} \nu_r\, t^r\quad.
\]

\begin{claim}
With notation as above, $\nu_X$ and $\pi_X$ are mapped to each other by
the involution $\cI$. 
\end{claim}

Indeed, the degree of the Milnor class equals the Parusi\'nski-Milnor number
(\cite{MR2001i:14009}, \S4.1); so this follows (as in the proof of Theorem~\ref{main})
from the fact that the Milnor class satisfies the same formula as $\csm$ does with 
respect to general hyperplanes sections. This in turn follows from the fact that, up to 
sign, the Milnor class equals the difference between the Chern-Schwartz-MacPherson
class and the Chern class of the virtual tangent bundle, which also behaves
as prescribed by Proposition~\ref{csmhypsec} with respect to hyperplane sections.
Details are left to the reader.

\begin{example}\label{arrangex}
As an example illustrating Theorem~\ref{main}, we consider a hyperplane
arrangement $\cA$ in $\Pbb^n$. With $\cA$, or more precisely with the corresponding
central arrangement $\hcA$ in $k^{n+1}$, we can associate the {\em characteristic 
polynomial\/} $P_\hcA(t)$ (Definition~2.5.2 in~\cite{MR1217488}). This is one of the 
most important combinatorial invariants of the arrangement; for example, in the case 
of graphical arrangements it recovers the chromatic polynomial of the corresponding 
graph (Theorem~2.88 in \cite{MR1217488}). It immediately follows from the definition 
that $P_\hcA(1)=0$, and we let $\underline P_\hcA(t)$ denote the quotient 
$P_\hcA(t)/(t-1)$. In general, $\underline P_{\cA}(t)$ agrees with the Poincar\'e 
polynomial of the complement $M(\cA)$ of $\cA$ in projective space up to a simple 
coordinate change (Theorem 5.93 in~\cite{MR1217488}).
\end{example}

\begin{corol}\label{hyparr}
With notation as above,
\[
\underline P_{\hcA}(t)=\frac{(t-1) \chi_{M(\cA)}(-t)+\chi_{M(\cA)}(0)}t
\quad\text{and}\quad
\chi_{M(\cA)}(t)=\frac{t\, \underline P_{\hcA}(-t)+\underline P_{\hcA}(1)}{t+1}\quad.
\]
\end{corol}

\begin{proof}
By Theorem~3.1 in~\cite{hyparr}, $\gamma_{M(\cA)}(t) = \underline P_{\hcA}(t+1)$.
Applying the involution~$\cI$ gives the first formula. The second formula is 
equivalent to the first.
\end{proof}

For example, consider the arrangement $\cA$ in $\Pbb^2$ consisting of three coincident
lines and of a line not containing the point of intersection:
\begin{center}
\includegraphics[scale=.6]{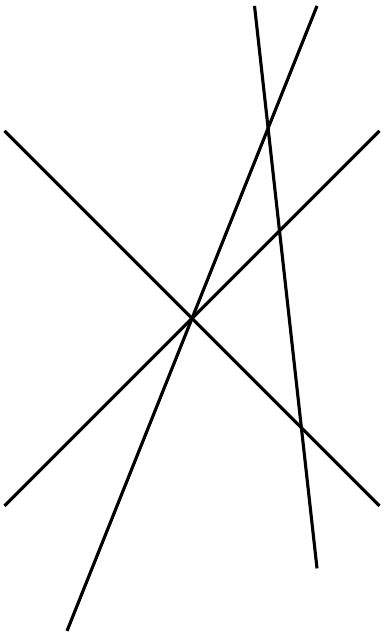}
\end{center}
It is immediately verified that $\chi_0(M(\cA))=0$, $\chi_1(M(\cA))=-2$, 
$\chi_2(M(\cA))=1$. (For instance, the intersection of $M(\cA)$ with a general
line consists of the complement of $4$ points in $\Pbb^1$, with Euler characteristic
$-2$.) Therefore $\chi_{M(\cA)}(t)=2t+t^2$, and hence
\[
\underline P_{\hcA}(t)=\frac{(t-1) (-2t+t^2)+0}{t}=t^2-3t+2\quad.
\]

Note that $\csm$ classes do not appear directly in the statement of 
Corollary~\ref{hyparr}, but they streamline the proof considerably, 
via~Theorem~\ref{main}.

J.~Huh has recently proved that the coefficients of $P_{\hcA}(t)$ form a 
{\em log-concave} sequence, settling long-standing conjectures of Read, Rota, 
Heron, and Welsh (\cite{Huh}). The relation of $P_{\hcA}(t)$ with information 
equivalent to the $\csm$ class of the complement (the {\em polar degrees,\/} 
cf.~\S\ref{DPH}, particularly~Remark~\ref{gchar}) plays an important r\^ole in 
Huh's work.
\qede


\section{Dimca-Papadima/Huh formulae}\label{DPH}

\subsection{Polar degrees}\label{polardegs}
The application sketched in \S\ref{introDPH} depends on the following result
from \cite{MR1956868}.

Let $X$ be a hypersurface of $\Pbb^n$, defined by a homogeneous polynomial
$F(x_0,\dots,x_n)$. Consider the rational map $\Pbb^n\dashrightarrow \Pbb^n$
defined by the partial derivatives of $F$:
\[
p \mapsto \left(\frac{\partial F}{\partial x_0}(p)\colon \cdots \colon 
\frac{\partial F}{\partial x_n}(p)\right)\quad.
\]
This is the {\em polar\/} map of \cite{MR1786486}, called the {\em gradient\/}
map in \cite{MR2018927}. We denote by $\Gamma$ the graph of this map.
The class of $\Gamma$ in $\Pbb^n\times \Pbb^n$ determines integers $g_i$
such that
\[
[\Gamma] = g_0 k^n + g_1 h k^{n-1} + \cdots + g_n h^n
\]
where $h$, resp., $k$ is the pull-back of the hyperplane class from the first,
resp., second factor. The number $g_i$ are the {\em projective degrees\/}
of the polar map (\cite{MR1182558}, Example~19.4): $g_i$ is the degree
of the restriction of the polar map to a general $\Pbb^i$ in $\Pbb^n$.
We will call $g_i$ the {\em $i$-th polar degree\/} of $F$. In particular, the 
$n$-th degree $g_n$ equals the degree of the polar map itself. By definition, 
the polynomial~$F$ is {\em homaloidal\/} if $g_n=1$, i.e., if the polar map 
is birational (\cite{MR1786486}).

All the considerations in this section will be consequences of Theorem~\ref{main}
and the following result.

\begin{theorem}[\cite{MR1956868}, Theorem 2.1]\label{cccps}
With notation as above,
\begin{equation}\label{thm02}
\csm(\one_X) = (1+h)^{n+1} - \sum_{j=0}^n g_j\cdot (-h)^j (1+h)^{n-j}\quad.
\end{equation}
\end{theorem}

The point is that $\Gamma$ may be identified with the blow-up of $\Pbb^n$
along the scheme defined by the ideal generated by the partial derivatives
of $F$, that is, the singularity subscheme $JX$ of $X$. It is then clear that
the class of $\Gamma$ carries essentially the same information as the
(push-forward of the) Segre class of $JX$ in $\Pbb^n$. Performing this
notational translation in the formula given in Theorem~\ref{1995thm}
yields Theorem~\ref{cccps}. Further details may be found in~\cite{MR1956868}.

The following is an immediate consequence of \eqref{thm02}:

\begin{corol}\label{doldim}
The polar degrees depend only on the reduced polynomial
$F_\text{red}$ associated with $F$. In particular, $F$ is homaloidal if 
and only if $F_{\text{red}}$ is homaloidal.
\end{corol}

Indeed, by means of Theorem~\ref{cccps} all the $g_j$'s are determined
by $\csm(\one_X)=\csm(\one_{X_{\text{red}}})$
(cf.~Remark~\ref{csmred}). 

As mentioned in the introduction, the last part of this statement had been conjectured 
by Dolgachev (\cite{MR1786486}, end of \S3). A topological proof (over $\Cbb$) 
was first given in~\cite{MR1946153}; several other proofs have appeared in 
the meanwhile.
The argument given above shows that the fact holds over any algebraically
closed field of characteristic zero, and is straightforward modulo the 
result from~\cite{MR2001i:14009} recalled in \S\ref{csmhyp}.

\subsection{Formularium}\label{formulae}
As mentioned above, the $g_i$ are the projective degrees of the polar map.
The term $(1+h)^{n+1}$ in \eqref{thm02} is the Chern class of $\Pbb^n$,
so Theorem~\ref{cccps} may be reformulated as
\begin{equation}\label{compleq}
\csm(\one_{\Pbb^n\smallsetminus X}) = \sum_{j=0}^n g_j\cdot (-h)^j (1+h)^{n-j}\quad.
\end{equation}
Using the notation introduced in \S\ref{proofmt} and applying Theorem~\ref{main} 
we get:

\begin{corol}\label{gachig}
Let $D(F)$ be the complement $\Pbb^n\smallsetminus X$. Then
\begin{itemize}
\item $\gamma_{D(F)}(t) = \sum_{j=0}^n g_j\cdot (-1)^j (t+1)^{n-j}$. 
\item $\chi_{D(F)}(t) = (-1)^n \sum_{j=0}^n g_j\cdot \dfrac{t^{n-j+1}-(-1)^{n-j+1}}{t+1}$.
\end{itemize}
\end{corol}

\begin{proof}
In \eqref{compleq}, $h^i$ stands for the class $h^i\cdot [\Pbb^n]=[\Pbb^{n-i}]$,
which is replaced by $t^{n-i}$ in the polynomial $\gamma_{D(F)}(t)$. This change
may be performed by replacing $h$ by $1/t$ and multiplying through by $t^n$,
yielding the first formula. The second follows immediately by applying the
involution $\cI$, as prescribed by Theorem~\ref{main}.
\end{proof}

We can assemble the $g_j$'s into yet another polynomial determined by $X$:
\[
g_X(t):= \sum_{j=0}^n g_j\, t^{n-j}\quad,
\]
and Corollary~\ref{gachig} may then be reformulated as
\begin{itemize}
\item $\gamma_{D(F)}(t) = (-1)^n\, g_X(-t-1)$. 
\item $\chi_{D(F)}(t) = (-1)^n\, \dfrac{t\cdot g_X(t)+g_X(-1)}{t+1}$.
\end{itemize}
`Solving for $g$' in these two formulas gives

\begin{corol}\label{gfromgamma}
With notation as above,
\begin{equation}\label{gfromchi}
g_X(t) = (-1)^n\, \gamma_{D(F)}(-t-1) 
= (-1)^n\, \frac{(t+1)\cdot \chi_{D(f)}(t) - \chi_{D(f)}(0)}t\quad.
\end{equation}
\end{corol}

\begin{remark}\label{gchar}
If $X$ is a hyperplane arrangement, Corollary~\ref{gfromgamma}
shows that $g_X(t)$ agrees with $\underline P_{\hcA}(-t)$ up to a sign
(cf.~Example~\ref{arrangex}).
This is observed (at least over $\Cbb$) in Corollary~25 of \cite{Huh}.

In general, the first formula shows that $\gamma_{D(f)}(t)$ and $g_X(t)$ are 
also related by an involution. 
\qede\end{remark}

In particular, \eqref{gfromchi} computes the Euler characteristic
of the complement $D(F)$ as
\[
\chi(D(F))=\gamma_{D(F)}(0)=(-1)^n g_X(-1) = g_0 -g_1+g_2 - \cdots \pm g_n
\quad.
\]
Further, the polar degree $g_n$ equals
\[
g_X(0)= (-1)^n \gamma_{D(f)}(-1)=(-1)^n \left( \chi_{D(f)}(0)+\chi_{D(f)}'(0)
\right)\quad,
\]
as was mentioned in~\S\ref{introDPH}. That is,
\begin{corol}[Dimca-Papadima]
The degree of the polar map determined by the homogeneous polynomial $F$
equals
\[
(-1)^n \left( \chi(D(F))-\chi(H\cap D(F))\right)
\]
where $H$ is a general hyperplane.
\end{corol}
Over $\Cbb$, this formula is given in \cite{MR2018927}, Theorem~1. The
argument presented above proves it over any algebraically closed field of
characteristic~$0$, where we adopt the definition of Euler characteristic $\chi$
recalled in~\S\ref{introEc}. Tracing the argument shows that
the precise requirement on $H$ is that it should intersect properly the 
supports of the normal cone of $JX$ (cf.~the proof of 
Proposition~\ref{csmhypsec}). This should be viewed as an algebro-geometric 
analog of the condition specified in the paragraph following the statement
of Theorem~1 in~\cite{MR2018927}.

More generally:
\begin{corol}[Huh]\label{Huhf}
For $j=0,\dots,n$:
\[
g_j=(-1)^j\, \chi(D(F)\cap (L_j\smallsetminus L_{j-1}))
\]
where $L_j$ is a general linear subspace of dimension~$j$ in $\Pbb^n$.
\end{corol}
\noindent (Just read off the coefficient of $t^{n-j}$ in \eqref{gfromchi}.)

Over $\Cbb$, this formula is given in Theorem~9 in~\cite{Huh}, where it is
obtained from a description of the homotopy type of the general
linear sections of $D(F)$. 

\begin{remark}
The connection between $\csm$ classes and the polar degrees $g_j$ is 
mentioned explicitly in~\cite{Huh}. Our only contribution here amounts to the
observation that this connection alone suffices for the formula in 
Corollary~\ref{Huhf}, modulo the rather simple-minded Theorem~\ref{main}.
This has the very minor advantage of providing a `non-topological' interpretation
of the formula, which is then
shown to hold over any algebraically closed field of characteristic~$0$.
\qede\end{remark}

\subsection{Higher codimension}
It is natural to ask whether formulas analogous to those reviewed in the
previous section may be given for higher codimensional subschemes 
in~$\Pbb^n$. It is not clear {\em a priori\/} what should play the role of
the `polar degrees' $g_j$ defined in \S\ref{polardegs}. Maybe the most
surprising aspect here is that one {\em can\/} in fact define these degrees
in complete generality.

Let $S\subseteq \Pbb^n$ be any subscheme, and let $F_1,\dots, F_r$
be nonzero homogeneous generators for any ideal $I$ defining $X$.
Denote by
$g_i^{(F)}$ the $i$-th polar degree of the homogeneous polynomial $F$,
defined as in \S\ref{polardegs}.

\begin{defin}\label{polardegsdef}
We define the {\em $i$-th polar degree of $S$\/} to be
\[
g_i^S:=\sum_{\emptyset\neq J\subseteq \{1,\dots,r\}}
(-1)^{|J|+1} g_i^{(\prod_{j\in J} F_j)}
\]
where $F_1,\dots, F_r$ are any collection of homogeneous polynomials
generating an ideal $I$ defining~$S$.
\qede\end{defin}

It is not obvious (to us) that these degrees are well-defined, that is, that they
do not depend on the choice of the generators of the ideal $I$. We will see that
they are, and that in fact they only depend on the support $S_\text{red}$ of
the scheme defined by~$I$. Thus, the numbers $g_i^S$ do not change if
$I$ is replaced by $\sqrt I$ or by the saturation of $I$. This fact generalizes
Dolgachev's conjecture to arbitrary subschemes of $\Pbb^n$, over any
algebraically closed field of characteristic~$0$.

\begin{example}\label{tcex}
The ideal of a twisted cubic $C\subseteq \Pbb^3$ is generated by the quadratic
polynomials $F_1=x_0 x_3-x_1 x_2$, $F_2=x_0 x_2-x_1^2$, $F_3=x_1 x_3-x_2^2$.
The polar map of $F_1$ is given in homogeneous coordinates by
$(x_3:-x_2:-x_1:x_0)$, giving $g^{(F_1)}_0 = \cdots = g^{(F_1)}_3 = 1$.
Both $F_2$ and $F_3$ are cones over smooth conics, and this gives easily
$g^{(F_i)}_0 = \cdots = g^{(F_i)}_2 = 1$, $g^{(F_i)}_3 = 0$ for $i=1,2$.
An explicit computation (which may be performed with e.g., Macaulay2
\cite{M2}) shows that 
\[
g^{(F_i F_j)}_0 = 1 ,\quad
g^{(F_i F_j)}_1 = 3 ,\quad
g^{(F_i F_j)}_2 = 5 ,\quad
g^{(F_i F_j)}_3 = 3 
\]
for $i\ne j$, and
\[
g^{(F_1 F_2 F_3)}_0 =  1,\quad
g^{(F_1 F_2 F_3)}_1 =  5,\quad
g^{(F_1 F_2 F_3)}_2 =  10,\quad
g^{(F_1 F_2 F_3)}_3 =  6\quad.
\]
It follows that
\[
g_0^{C}=1,\quad
g_1^{C}=-1,\quad
g_2^{C}=-2,\quad
g_3^{C}=-2\quad.
\]
(For example, $g_3^C=1+0+0-3-3-3+6=-2$.)

On the other hand, the twisted cubic $C$ is also the set-theoretic intersection 
of the quadric $F=x_0 x_2-x_1^2=0$ and the cubic $G=x_2 (x_1 x_3-x_2^2)
-x_3 (x_0 x_3-x_1 x_2)=0$. An explicit computation (again performed with
Macaulay2) gives
\begin{gather*}
g^{(F)}_0 = 1 ,\quad
g^{(F)}_1 = 1 ,\quad
g^{(F)}_2 = 1 ,\quad
g^{(F)}_3 = 0 \\
g^{(G)}_0 = 1 ,\quad
g^{(G)}_1 = 2 ,\quad
g^{(G)}_2 = 3 ,\quad
g^{(G)}_3 = 1
\end{gather*}
(so that $G$ is homaloidal; this plays no r\^ole here) and
\[
g^{(FG)}_0 = 1 ,\quad
g^{(FG)}_1 = 4 ,\quad
g^{(FG)}_2 = 6 ,\quad
g^{(FG)}_3 = 3 \quad.
\]
This gives
\[
g_0^{C}=1+1-1=1,\quad
g_1^{C}=1+2-4=-1,\quad
g_2^{C}=1+3-6=-2,\quad
g_3^{C}=0+1-3=-2\,,
\]
with the same result from the ideal for a different scheme structure, as promised.
\qede\end{example}

To prove that the polar degrees of a subscheme $S$ are well-defined, and in 
fact only depend on $S_\text{red}$, it suffices to observe that they are related
with the polynomials $\gamma_{\Pbb^n\smallsetminus S}$, $\chi_{\Pbb^n
\smallsetminus S}$ by {\em the same formulas as
in the hypersurface case.\/} As in the hypersurface case (but now for arbitrary 
subschemes $S\subseteq \Pbb^n$) we define
\[
g_S(t):= \sum_{i=0}^n g^S_i\, t^{n-i}\quad.
\]

\begin{theorem}\label{arbcod}
With notation as above,
\[
g_S(t) = (-1)^n\, \gamma_{\Pbb^n\smallsetminus S}(-t-1) 
= (-1)^n\, \frac{(t+1)\cdot \chi_{\Pbb^n\smallsetminus S}(t) 
- \chi_{\Pbb^n\smallsetminus S}(0)}t\quad.
\]
\end{theorem}

This proves that Definition~\ref{polardegsdef} is indeed independent of the
ideal chosen to define~$S$, or of the generators of this ideal (since the 
other expressions are independent of these choices), and that
$g^S_i=g^{S_\text{red}}_i$, since $S$ and $S_\text{red}$ have the same
complement in~$\Pbb^n$.

Of course the other formulas encountered in~\S\ref{formulae} also hold
for arbitrary $S$, since they may be derived from the equalities given in 
Theorem~\ref{arbcod}. Thus, 
\[
\chi(\Pbb^n\smallsetminus S) =(-1)^n g_S(-1)\quad,
\]
and
\[
g^S_j =(-1)^j \chi((\Pbb^n\smallsetminus S)\cap (L_j\smallsetminus L_{j-1}))
\]
as in Huh's formulas for hypersurfaces (Corollary~\ref{Huhf}).
For instance, with $C$ the twisted cubic as in Example~\ref{tcex},
the intersection of $\Pbb^3\smallsetminus C$ with a general $\Pbb^2
\smallsetminus \Pbb^1$ consists of the complement in $\Pbb^2$
of a general line and $3$ distinct points, hence
\[
g_2^C = (-1)^2 (\chi(\Pbb^2)-\chi(\Pbb^1)-3\,\chi(\Pbb^0)) = -2
\]
in agreement with the algebraic computation(s) given in Example~\ref{tcex}.

\begin{proof}[Proof of Theorem~\ref{arbcod}]
Choose any collection of generators $F_1,\dots, F_r$ for any ideal~$I$ defining 
$S$, as in Definition~\ref{polardegsdef}, and define $g_S(t)$ as specified above.
Also, let $X_i$ be the hypersurface of $\Pbb^n$ defined by $F_i$.
Then
\begin{align*}
(-1)^n\, g_S(-t-1)
&=(-1)^n \sum_{i=0}^n \sum_{\emptyset\ne J\subseteq \{1,\dots,r\}} 
(-1)^{|J|+1} g_i^{(\prod_{j\in J} F_j)} (-t-1)^{n-i} \\
&=\sum_{\emptyset\ne J\subseteq \{1,\dots,r\}} (-1)^{|J|+1}
(-1)^n\, g_{\cup_{j\in J} X_j}(-t-1) \\
&=\sum_{\emptyset\ne J\subseteq \{1,\dots,r\}} (-1)^{|J|+1}
\gamma_{\Pbb^n\smallsetminus (\cup_{j\in J} X_j)} (t)
\end{align*}
by Corollary~\ref{gfromgamma}. Now, $\gamma_{\Pbb^n\smallsetminus 
(\cup_{j\in J}X_j)} (t)$ is the polynomial corresponding to the $\csm$ class of
$\one_{\Pbb^n\smallsetminus (\cup_{j\in J}X_j)}$. Therefore, the end result
is the polynomial corresponding to the $\csm$ class of the constructible
function
\[
\sum_{\emptyset\ne J\subseteq \{1,\dots,r\}} (-1)^{|J|+1}
\one_{\Pbb^n\smallsetminus (\cup_{j\in J} X_j)}=
\sum_{\emptyset\ne J\subseteq \{1,\dots,r\}} (-1)^{|J|+1}
\one_{\cap_{j\in J}(\Pbb^n\smallsetminus X_j)}\quad.
\]
A simple inclusion-exclusion argument shows that this equals
\[
\one_{\cup_{j=1,\dots,r} (\Pbb\smallsetminus X_j)}=\one_{\Pbb^n\smallsetminus S}
\quad.
\]
Therefore,
\[
(-1)^n g_S(-t-1)=\gamma_{\Pbb^n\smallsetminus S}(t)\quad,
\]
or equivalently
\[
g_S(t) = (-1)^n\, \gamma_{\Pbb^n\smallsetminus S}(-t-1)\quad.
\]
The other equality in Theorem~\ref{arbcod} follows from this, by applying 
Theorem~\ref{main}.
\end{proof}

It would be desirable to have a more direct argument showing the independence 
of the degrees $g_i^S$ on the choices used in Definition~\ref{polardegsdef} to 
define them. (This is a reformulation of a problem posed in~\cite{MR2007377}.)
Example~\ref{tcex} shows that the polar degrees of a higher codimension 
subscheme may be negative; in particular, they cannot be directly interpreted 
as degrees of rational maps as in the hypersurface case. It seems conceivable
that they can be expressed as Euler characteristics of complexes determined
by the ideal sheaf of $S$. 


\section{Polynomial Chern-Schwartz-MacPherson classes}\label{EC}

\subsection{}
We now switch our focus to a different question. Our goal is to provide
the gist of a theory of characteristic classes for projective varieties
(over our fixed algebraically closed field of characteristic~$0$), with values in $\Zbb[t]$. 
Such a theory can be constructed using the theory of Chern-Schwartz-MacPherson 
classes (Proposition~\ref{cstCSM}), but we are going to take a naive approach
and not assume the existence of $\csm$ classes. We find it remarkable that
the `polynomial' version of this theory, including an analogue of the key covariance
property of $\csm$ classes, can be established using no tools other than the
naive considerations on Euler characteristics recalled in \S\ref{introEc}.

\subsection{$\Pbb^{\infty}$-varieties}\label{Pinftyvars}
Our objects will be varieties endowed with base-point-free linear systems.
We consider a fixed infinite chain of embeddings
\[
\Pbb^0 \subseteq \Pbb^1 \subseteq \Pbb^2 \subseteq \cdots \quad,
\]
as a direct system; $\iota_{nm}$ is the chosen embedding $\Pbb^m \to \Pbb^n$ for 
$n>m$. We denote by $\Pbb^\infty$ the limit of this system.
A `$\Pbb^\infty$-variety' is represented by a regular morphism $\varphi: X\to \Pbb^m$, 
where we identify $\varphi: X\to \Pbb^m$ with $\iota_{nm}\circ \varphi: X\to \Pbb^n$ 
for $m<n$. In particular, we may always assume that any two $\Pbb^\infty$-varieties
are represented by morphisms with common target: $\varphi: X\to \Pbb^n$, 
$\psi: Y \to \Pbb^n$; a morphism between the corresponding $\Pbb^\infty$-varieties
is then represented by a commutative diagram
\[
\xymatrix@C=10pt{
X \ar[rr]^f \ar[d]_\varphi & & Y \ar[d]^\psi \\
\Pbb^n \ar[rr]_\xi^\sim & & \Pbb^n
}
\]
such that $f$ is a regular morphism and $\xi$ is an isomorphism. 
Thus, a $\Pbb^\infty$-variety is a variety endowed with a morphism to a projective 
space, and isomorphisms of $\Pbb^\infty$-varieties are induced by `automorphisms
of $\Pbb^\infty$' (by which we mean automorphisms induced by automorphisms at 
some finite level).
For example, a line $L\hookrightarrow \Pbb^2$ and a conic $C\hookrightarrow \Pbb^2$
are {\em not\/} isomorphic as $\Pbb^\infty$-varieties, since the abstract isomorphism
$L\cong \Pbb^1\cong C$ is not induced by an automorphism of $\Pbb^2$; in other
words, the corresponding linear systems do not match.

We also consider a `Chow group' $A_*\Pbb^\infty$, as the direct limit $\varinjlim A_*\Pbb^n$; 
this is the free abelian group on classes $[\Pbb^i]$ for $i\ge 0$. We identify 
$A_*\Pbb^\infty$ with $\Zbb[t]$, where $t^i$ is associated with $[\Pbb^i]$. Of course 
the {\em ring\/} structure of $\Zbb[t] = A_*\Pbb^\infty$ does not define an intersection 
product, but it serves as a convenient shorthand for manipulations of classes in 
$A_*\Pbb^\infty$. We will also relate it with products in a Grothendieck ring of 
$\Pbb^\infty$-varieties in~\S\ref{GrR}.

If $\varphi$ is proper, then push-forward to $\Pbb^n$ followed by the inclusion in 
the direct limit defines a push-forward $\varphi_*: A_*X \to \Zbb[t]$ for every object 
$\varphi: X\to \Pbb^n$. Concretely, for $a\in A_*X$, the coefficient of $t^k$ in 
$\varphi_* (a)\in \Zbb[t]$ equals $\int (\varphi^*(H))^k\cdot a$, where $H$ is the 
hyperplane class. This push-forward satisfies the evident compatibility property with 
morphisms: if $(f,\xi)$ defines a morphism $\varphi \to \psi$ as above and all maps are 
proper, then $\varphi_* =  \psi_*\circ f_*$. (Indeed, automorphisms of $\Pbb^\infty$ 
induce the identity on $A_*\Pbb^\infty$.)

\subsection{Chern classes in $A_*\Pbb^\infty$}\label{numChcl}
We associate with $\varphi: X \to \Pbb^n$ the group of constructible functions
$\cC(X)$, and recall that $\cC$ is a covariant functor, see~\S\ref{introCSM}. 
For $\alpha\in \cC(X)$, we seek a `Chern class' 
$c^\varphi_*(\alpha)\in A_*\Pbb^\infty=\Zbb[t]$ with the following properties:
\begin{itemize}
\item[(i)] $c^\varphi_*$ is a group homomorphism $\cC(X) \to \Zbb[t]$;
\item[(ii)] If $\varphi$ is proper and $X$ is nonsingular, then 
$c^\varphi_*(\one_X)=\varphi_*(c(TX)\cap [X])$;
\item[(iii)] If $(f,\xi)$ defines a morphism $\varphi\to \psi$, and $\alpha\in \cC(X)$,
then $c^\varphi_*(\alpha) = c^\psi_*(f_* (\alpha))$.
\end{itemize}

By resolution of singularities, a theory satisfying requirements (i)--(iii) is necessarily
unique. We could use $\csm$ classes to provide such 
a notion (cf.~Proposition~\ref{cstCSM} below); but the work involved in proving 
the existence of $\csm$ classes is itself nontrivial. We want to advertise an 
alternative, simpler construction, suggested by Theorem~\ref{main}.

By linearity, it suffices to define $c_*^\varphi(\one_Z)$, for an object $\varphi:
X\to \Pbb^n$ and a closed subvariety $Z$. Given such data, we let 
$\chi^\varphi_i(Z)$ denote the Euler characteristic (in the sense of \S\ref{introEc}) 
of $\varphi^{-1}(L)\cap Z$ for a general linear subspace $L\subseteq \Pbb^n$ of 
codimension $i$. We let
\[
\chi^\varphi_Z(t):= \sum_{r\ge 0} (-1)^r \chi^\varphi_i(Z)\, t^i\quad,
\]
and note that this is compatible with the definition given in \S\ref{introsu}, 
to which it reduces if $\varphi$ is an embedding.

\begin{defin}\label{numdef}
We define $c_*^\varphi(\one_Z)$ to be $\cI(\chi^\varphi_Z)\in\Zbb[t]
=A_*\Pbb^\infty$, 
where $\cI$ is the involution defined in \S\ref{introsu}. Explicitly,
\[
c_*^\varphi(\one_Z)=\frac{t\cdot \chi^\varphi_Z(-t-1)+\chi^\varphi_Z(0)}{t+1}\quad.
\]
\end{defin}

\begin{example}
The constant term of $c_*^\varphi(\one_Z)$ equals $\chi(Z)$. 
Indeed, $c_*^\varphi(\one_Z)|_{t=0} = \chi_Z^\varphi (0) = \chi_0^\varphi(Z)=\chi(Z)$.
This is as it should be expected, given the parallel between the characterizing 
properties (i)--(iii) for $c_*^\varphi$ and the Deligne-Grothendieck axioms for $\csm$
classes (cf.~\eqref{degcsm}).
\qede\end{example}

We now proceed to verifying properties (i)--(iii). Property (i) is implicit in the construction.
Property (ii):

\begin{lemma}[Normalization]\label{normpro}
If $\varphi$ is proper and $X$ is nonsingular, then
\[
c^\varphi_*(\one_X)
=\varphi_*(c(TX)\cap [X])\quad.
\]
\end{lemma}

\begin{proof}
By Bertini's theorem (Corollary~10.9 in~\cite{MR0463157}), if $L$ is a general
codimension~$i$ subspace of $\Pbb^n$, then $\varphi^{-1}(L)$ is a codimension~$i$
nonsingular subvariety of $X$. The class $[\varphi^{-1}(L)]$ equals $H^i\cdot [X]$,
where $H$ is the pull-back of the class of a hyperplane, and the normal bundle
$N_{\varphi^{-1}(L)}X$ is (the restriction of) $(1+H)^i$. Thus
\[
\chi^\varphi_i(Z) = \int c(T(\varphi^{-1}(L)))\cap [\varphi^{-1}(L)] 
=\int \frac{H^i}{(1+H)^i}\, c(TX)\cap [X]\quad.
\]
Now we argue exactly as in the proof of Theorem~\ref{main}, and obtain
\[
\chi^\varphi_X(t) =\sum_{r\ge 0} \left(\int \frac{(-H)^r}{(1+H)^r}\, c(TX)\cap [X]\right)t^r 
=c_0-t\sum_{\ell\ge 1} c_\ell \,(-t-1)^{\ell-1}
\]
where $c_i=\int H^i\cdot c(TX)\cap [X]$. As $c_0+c_1 t+\cdots+c_n t^n 
= \varphi_*(c(TX)\cap [X])$, this says
\[
\chi_X^\varphi=\cI(\varphi_* c(TX)\cap [X])\quad,
\]
and it follows that $\cI(\chi_X^\varphi)=\varphi_* c(TX)\cap [X]$ as $\cI$ is an involution.
This is precisely the statement.
\end{proof}

Property (iii):

\begin{lemma}[Covariance]\label{covarlem}
Let 
\[
\xymatrix@C=10pt{
X \ar[rr]^f \ar[d]_\varphi & & Y \ar[d]^\psi \\
\Pbb^n \ar[rr]_\xi^\sim & & \Pbb^n
}
\]
be a commutative diagram, and let $\alpha\in \cC(X)$.
Then $c^\varphi_*(\alpha) = c^\psi_*(f_* (\alpha))$.
\end{lemma}

\begin{proof}
Since $\xi$ is an isomorphism, it is clear that $c_*^{\xi\circ \varphi}(\alpha)
=c_*^\varphi(\alpha)$; thus we may assume that $\xi$ is the identity, and 
$\varphi=\psi\circ f$.

By linearity we may assume $\alpha=\one_Z$, with $Z\subseteq X$ a closed 
subvariety. Apply Lemma~\ref{fiberlemma} to $f|_Z$ to deduce the existence
of subvarieties $V_1,\dots,V_r$ of $Y$ and integers $m_1,\dots,m_r$
such that $f_*(\one_Z)=\sum_j m_j\one_{V_j}$. 
Then 
\[
c^\psi_*(f_* (\one_Z))= c^\psi_*(\sum_j m_j \one_{V_j}) 
= \sum_j m_j c^\psi_*(\one_{V_j})=\sum_j m_j \cI(\chi^\psi_{V_j})
=\cI(\sum_j m_j \chi^\psi_{V_j})\quad.
\]
On the other hand, $c_*^\varphi(\one_Z)=\cI(\chi^\varphi_Z)$.
This shows that the equality $c^\varphi_*(\alpha) = c^\psi_*(f_* (\alpha))$ is
equivalent to the statement
\[
\chi^\varphi_Z(t) = \sum_j m_j \chi^\psi_{V_j}(t)
\]
and hence to
\[
\forall i\quad,\quad \chi^\varphi_i(Z) = \sum_j m_j \chi^\psi_i(V_j)\quad.
\]
For each $i$, the left-hand side is
\[
\chi(\varphi^{-1}(L)\cap Z)=\chi( f|_Z^{-1}(\psi^{-1}(L)))
\]
where $L$ is a general subspace of $\Pbb^n$ of codimension~$i$. By the second
part of Lemma~\ref{fiberlemma}, this equals
\[
\sum_{j=1}^r m_j \chi(\psi^{-1}(L)\cap V_j)=\sum_{j=1}^r m_j \chi^\psi_i(V_j)\quad,
\]
concluding the proof.
\end{proof}

\subsection{Polynomial Chern classes and Grothendieck ring(s)}\label{GrR}
The free abelian group of isomorphism classes of $\Pbb^\infty$-varieties modulo 
the usual scissor relations defines a `relative Grothendieck group of varieties
over $\Pbb^\infty$', which we will denote $K(\Var_{\Pbb^\infty})$.
If $\varphi: X \to \Pbb^m$ is an (understood) embedding, we write $[X]$ 
for the corresponding element $[\varphi]\in K(\Var_{\Pbb^\infty})$. The part 
of the Grothendieck group determined by embeddings is essentially the same
as the Grothendieck group of `immersed conical varieties' studied in 
\cite{MR2782886}. Also note that $K(\Var_{\Pbb^\infty})$ admits generators
factoring through affine space: $\varphi: X\to \Pbb^m$ obtained by composing
a morphism $\varphi^\circ: X\to \Abb^m$ with a standard embedding into $\Pbb^m$.
It also admits a description in terms of Bittner's relations. 

Isomorphisms of $\Pbb^\infty$ varieties allow for automorphisms of the base 
$\Pbb^\infty$; the usual context of relative Grothendieck groups
(as in e.g., \cite{MR2059227}, \S5) does not. This appears to be advantageous
here since then this Grothendieck group carries more interesting products. For example, 
we can define a product by specifying the operation on generators, as follows:
if $[\varphi], [\psi]\in K(\Var_{\Pbb^\infty})$ are represented by
$\varphi: X\to \Pbb^{m-1}$ and $\psi: Y \to \Pbb^{n-1}$, define $[\varphi]\star [\psi]$ 
to be the class represented by the morphism
\[
\xymatrix{
X\times Y \ar[r]^-{\varphi\times \psi} & 
\Pbb^{m-1}\times \Pbb^{n-1} \ar@{^(->}[r]^-s & 
\Pbb^{mn-1}
}
\]
where $s$ is the Segre embedding. This product is distributive and associative up to 
automorphisms of $\Pbb^\infty$, so it defines a ring structure on $K(\Var_{\Pbb^\infty})$.
A different ring structure will be defined in \S\ref{prod}. 
Determining the precise behavior of $c_*$ with respect to these (and possibly 
other) products is an interesting problem, as knowledge of this behavior is helpful
in concrete computations.

The Chern classes defined in~\S\ref{numChcl} factor through $K(\Var_{\Pbb^\infty})$, 
in the sense that $c_*^\varphi(\alpha)$ only depends on the isomorphism class of 
$\varphi$ and this assignment satisfies the relations defining $K(\Var_{\Pbb^\infty})$. 
That is, if we have an object $\varphi: X\to \Pbb^n$, a closed
subvariety $i: Z\hookrightarrow X$, and let $j: U=X\smallsetminus Z 
\hookrightarrow X$ be the complement, then
\[
c_*^\varphi(\one_X) = c_*^{\varphi\circ i}(\one_Z) + c_*^{\varphi\circ j}(\one_U)\quad.
\]
Indeed, this identity may be verified after applying the involution $\cI$, which
gives
\[
\chi_X^\varphi(t) = \chi_Z^{\varphi\circ i}(t) + \chi_U^{\varphi\circ j}(t)\quad;
\]
and this latter identity is immediate from the additivity of Euler characteristic on 
disjoint unions.

\begin{prop}\label{add}
The assignment $(\varphi: X\to \Pbb^n) \mapsto c_*^\varphi(\one_X)$ defines a 
group homomorphism $\gamma: K(\Var_{\Pbb^\infty}) \to \Zbb[t]$.
\end{prop}

The same assignment has an interesting behavior with respect to products.
Consider the $\Zbb$-module automorphism $\sigma: \Qbb[t] \to \Qbb[t]$ defined 
on generators by $t^i \mapsto \frac{t^i}{i!}$. (This homomorphism could be 
defined for power series; in the terminology of combinatorics, it turns `ordinary' 
generating functions into `exponential' ones.) 

\begin{prop}\label{mult}
$\sigma\circ \gamma: (K(\Var_{\Pbb^\infty}),\star) \to \Qbb[t]$ is a {\em ring\/}
homomorphism.
\end{prop}

\begin{proof}
To verify that $\sigma\circ \gamma$ preserves products, we use Bittner's relations 
(\cite{MR2059227}, Remark~3.2): it suffices to verify that if $\varphi: X\to \Pbb^m$ 
and $\psi: Y \to \Pbb^n$ are morphisms with $X$ and $Y$ projective and nonsingular,
then $\sigma(c_*^{s\circ (\varphi\times \psi)}(\one_{X\times Y}))=
\sigma(c_*^\varphi(\one_X))\,\sigma(c_*^\psi(\one_Y))$. This will follow from the 
normalization property of $c_*$. Denote by $h_1$, resp., $h_2$, $H$ the 
hyperplane class in $\Pbb^{m-1}$, resp., $\Pbb^{n-1}$, $\Pbb^{mn-1}$. 
Since $X$ and $Y$ are nonsingular,
\[
c_*^\varphi(\one_X)=\sum_i c'_i t^i ,\quad
c_*^\psi(\one_Y)=\sum_j c''_j t^j
\]
where by Lemma~\ref{normpro} $c'_i=\int (\varphi^*h_1)^i\cdot c(TX)\cap [X]$
and $c''_j=\int (\psi^*h_2)^j\cdot c(TY)\cap [Y]$. Likewise,
\[
c_*^{s\circ (\varphi\times \psi)}(\one_{X\times Y})=\sum_k c_k t^k
\]
where $c_k = \int ((s\circ (\varphi\times \psi)^*H)^k \cdot c(T(X\times Y))\cap [X\times Y]$.
We have $(s\circ (\varphi\times \psi))^*H=(\varphi\circ p_1)^*h_1+(\psi\circ p_2)^*h_2$, 
where $p_1$, resp., $p_2$ is the first, resp., second projection from $X\times Y$. Also, 
$T(X\times Y) \cong p_1^*TX \oplus p_2^*TY$. It follows that
\begin{align*}
c_k &= \int ((s\circ (\varphi\times \psi))^*H^k \cdot c(T(X\times Y))\cap [X\times Y] \\
&= \int ((\varphi\circ p_1)^*h_1+(\psi\circ p_2)^*h_2)^k \cdot p_1^* c(TX) \cap p_2^* c(TY)
\cap [X\times Y] \\
&=\sum_{i+j=k} \binom ki \int p_1^* (\varphi^* h_1^i\cdot c(TX))\cap 
p_2^* (\psi^* h_2^j\cdot c(TY)) \cap [X\times Y] \\
&=\sum_{i+j=k} \binom ki \left(\int \varphi^* h_1^i\cdot c(TX)\cap [X]\right)
\left(\int \psi^* h_2^j\cdot c(TY)\cap [Y]\right) \\
&=\sum_{i+j=k} \binom ki c'_i c''_j \quad,
\end{align*}
and hence
\[
\sigma(c_*^{s\circ (\varphi\times \psi)}(\one_{X\times Y}))=
\sum_k \sum_{i+j=k} \binom ki c'_i c''_j \frac{t^k}{k!} 
=\sum_k \sum_{i+j=k} c'_i c''_j \frac{t^i}{i!} \frac {t^j}{j!}
=\sigma(c_*^\varphi(\one_X))\,\sigma(c_*^\psi(\one_Y))
\]
as needed.
\end{proof}

\begin{example}\label{P1xP1}
Consider the class $[\Pbb^1]\in K(\Var_{\Pbb^\infty})$ determined by the identity: 
$\Pbb^1 \to \Pbb^1$. Then $[\Pbb^1]\star [\Pbb^1]=[\Pbb^1\times \Pbb^1]$, 
where $\iota: \Pbb^1\times \Pbb^1\hookrightarrow \Pbb^3$ embeds $\Pbb^1 
\times\Pbb^1$ as a nonsingular quadric $Q$ in $\Pbb^3$. The polynomial
Chern class of $\Pbb^1$ is $2+t$. According to Proposition~\ref{mult},
\[
\sigma (c^\iota_*(\one_{\Pbb^1\times\Pbb^1})) = (\sigma(2+t))^2
=(2+t)^2 =4+4t+t^2\quad.
\]
Therefore, 
\[
c^\iota_*(\one_{\Pbb^1\times\Pbb^1})= \sigma^{-1}(4+4t+t^2)=4+4t+2t^2\quad.
\]
This is as it should: for a nonsingular quadric in $\Pbb^3$ the Euler characteristics
of general linear sections are $\chi_0=4,\chi_1=2,\chi_2=2$, and therefore
\[
c^\iota_*(\one_{\Pbb^1\times\Pbb^1}) = \cI(4-2t+2t^2)=4+4t+2t^2
\]
according to Definition~\ref{numdef}.
\qede\end{example}

There is at least one alternative product that may be considered on 
$K(\Var_{\Pbb^\infty})$, see~\S\ref{prod}. 

\subsection{Polynomial Chern classes and $\csm$ classes}
The characterizing properties listed in~\S\ref{numChcl} imply that the classes are a 
numerical aspect of the Chern-Schwartz-MacPherson classes:

\begin{prop}\label{cstCSM}
If $\varphi$ is proper, then $c^\varphi_*(\alpha)=\varphi_*(\csm(\alpha))$.
\end{prop}

Indeed, the normalizations of $c_*$ and $\csm$ are compatible, so it suffices
to verify that the covariance property of $\csm$ classes implies the third property 
of the classes defined in~\S\ref{numChcl}. If $f: X \to Y$ is proper, 
then $\forall \alpha \in \cC(X)$
\[
\psi_*(\csm(f_*\alpha))=\psi_*(f_*(\csm(\alpha)))=\varphi_* \csm(\alpha)
\]
as needed. Note that with notation as in \S\ref{proofmt} and Proposition~\ref{add}, 
we have $\gamma([X])=\gamma_X(t)$.

Of course we could use the formula in Proposition~\ref{cstCSM} to provide a 
construction of the classes $c_*^\varphi$ alternative to the one given in 
Definition~\ref{numdef}. Similarly, we could use the Grothendieck group 
to define $c_*^\varphi$: if $X$ is nonsingular and $\varphi: X\to \Pbb^n$ is 
proper, we could {\em define\/} $c_*^\varphi(\one_X)$ to be 
$\varphi_*(c(TX)\cap [X])$; and then use Bittner's relations to show this prescription 
descends to the Grothendieck group (cf.~Lemma~\ref{Bittner}), giving a notion for
possibly singular or noncomplete sources. The normalization and covariance 
properties would be immediate; in particular, the resulting class must agree with 
the one given in Definition~\ref{numChcl}.

\begin{remark}
By Proposition~\ref{cstCSM}, the algorithm presented in \cite{MR1956868} computes 
the polynomial Chern class
of a subscheme of projective space, given generators for an ideal defining it 
set-theoretically. By means of the definition given here (Definition~\ref{numdef}),
any algorithm computing Euler characteristics may be adapted to compute the
polynomial Chern class. One such algorithm is presented by Marco-Buzun\'ariz
in~\cite{marco}, together with a polynomial generalization of the Euler characteristic. 
This generalization differs from the polynomial Chern class introduced above by a 
simple change of coordinates, as follows from Proposition~\ref{cstCSM} and the 
result proved by Rennemo in the appendix to~\cite{marco}.
\qede\end{remark}

The elementary approach described in \S\ref{numChcl} streamlines the proof of 
some properties of these polynomial $\csm$ classes. For example, 
let $\iota: X\hookrightarrow \Pbb^n$ be a projective variety, and consider the
cone $\iota': X' \hookrightarrow \Pbb^{n+1}$ with vertex a point. It is clear that
$\chi(X')=1+\chi(X)$ and (with notation as in \S\ref{numChcl}) 
$\chi_j^{\iota'}(X')=\chi_{j-1}^\iota(X)$ for $j>0$. That is,
\[
\chi^{\iota'}_{X'}(t)=1+\chi(X)-t \chi^\iota_X(t)\quad,
\]
and hence
\[
c_*^{\iota'}(\one_{X'})=\cI(\chi(X)+1-t \chi^\iota_X(t))=\chi(X)+1+t \chi^\iota_X(-t-1)
=(t+1) c_*^\iota(t)+1\quad.
\]
This formula matches the one obtained in Proposition~5.2 in~\cite{MR2504753}
by a somewhat more involved argument. 

On the other hand, the relation obtained in Proposition~\ref{cstCSM} allows us
to interpret results for $\csm$ classes in terms of their polynomial aspect. 
We illustrate one application in the following, final section.

\subsection{Another ring homomorphism}\label{prod}
We can give a different, and in a sense more natural, multiplication operation on
$K(\Var_{\Pbb^\infty})$. We define it on {\em affine\/} generators, i.e., 
morphisms $X\to \Pbb^m$, $Y\to \Pbb^n$ which factor through affine space:
\[
\xymatrix{
X \ar@/_1pc/[rr]_\varphi \ar[r]^{\varphi^0} & \Abb^m \ar@{^(->}[r] & \Pbb^m & 
Y \ar@/_1pc/[rr]_\psi \ar[r]^{\psi^0} & \Abb^n \ar@{^(->}[r] & \Pbb^n
}
\]
We set $[\varphi]\cdot [\psi]$ to be the class represented by the morphism
\[
X\times Y \to \Abb^{m+n} \hookrightarrow \Pbb^{m+n}
\]
defined by
\begin{multline*}
(x,y) \mapsto (\varphi_1^\circ(x),\dots, \varphi_m^\circ(x),
\psi_1^\circ(x),\dots, \psi_n^\circ(x)) \\
\mapsto
(1\colon \varphi_1^\circ(x)\colon \dots\colon \varphi_m^\circ(x)\colon
\psi_1^\circ(x)\colon\dots\colon \psi_n^\circ(x))\quad.
\end{multline*}

This definition has a counterintuitive aspect to it: Although it does determine 
(by distributivity) a class $[\varphi]\cdot [\psi]$ for every $[\varphi], [\psi]\in
K(\Var_{\Pbb^\infty})$, this class may not have a compelling geometric
realization. For example, we have
\begin{align*}
[\Pbb^1]\cdot [\Pbb^1] &= ([\Abb^1]+[\Abb^0])\cdot ([\Abb^1]+[\Abb^0])
=[\Abb^1]\cdot [\Abb^1] + 2 [\Abb^1]\cdot [\Abb^0] + [\Abb^0]\cdot [\Abb^0] \\
&=[\Abb^2] + 2[\Abb^1] + [\Abb^0]\quad,
\end{align*}
but this does {\em not\/} equal $[\Pbb^1 \times \Pbb^1]=[\Pbb^1]\star [\Pbb^1]$
in $K(\Var_{\Pbb^\infty})$: although $\Pbb^1\times \Pbb^1$ admits the same
affine decomposition, the morphisms induced on the affine pieces by the Segre
map are not the inclusions as dense open sets of the corresponding projective
spaces. 

\begin{remark}\label{join}
A partial antidote to this unpleasant feature is through
the {\em join\/} construction. If $X\subseteq \Pbb^m$ and $Y\subseteq \Pbb^n$,
then we can place $\Pbb^m$ and $\Pbb^n$ as disjoint subspaces of $\Pbb^{m+n+1}$
(by acting with an `automorphism of $\Pbb^\infty$' on $Y\subseteq \Pbb^n
\subseteq \Pbb^{m+n+1}$),
and let $J(X,Y)$ be the union of the lines joining points of $X$ to points of $Y$.
The complement $J(X,Y)^\circ$ of $X$ and $Y$ in $J(X,Y)$ maps surjectively to 
$X\times Y$, with $k^*$ fibers. The reader can verify that
\[
[J(X,Y)^\circ] = \Tbb\cdot [X]\cdot [Y]
\]
in $K(\Var_{\Pbb^\infty})$, where $\Tbb$ denotes the class of the natural embedding
$k^*\subseteq \Pbb^1$. Thus, while $[X]\cdot [Y]$ may not have a direct `geometric'
realization, $[X]\cdot [Y]\cdot \Tbb$ does.
\qede\end{remark}

The new product $\cdot$ clearly defines an alternative structure of ring on the group 
$K(\Var_{\Pbb^\infty})$.
Now recall that we have defined a group homomorphism $\gamma: K(\Var_{\Pbb^\infty})
\to \Zbb[t]$, cf.~Proposition~\ref{add}.

\begin{prop}\label{mult2}
$\gamma$ is a {\em ring\/} homomorphism $(K(\Var_{\Pbb^\infty}),\cdot) \to \Zbb[t]$.
\end{prop}

\begin{proof}
We have to verify that if $\varphi: X \to \Pbb^{m-1}$, $\psi: Y \to \Pbb^{n-1}$ are affine
generators, then $c_*^{\varphi\times \psi}(\one_{X\times Y})=
c_*^\varphi(\one_X) c_*^\psi(\one_Y)$. By Lemma~\ref{covarlem}, this is easily
reduced to the case in which $\varphi$, $\psi$ are embeddings. 
Using Remark~\ref{join}, we see that it suffices to verify that 
\begin{equation}\label{homom}
\gamma([J(X,Y)^\circ])=\gamma(\Tbb)\, \gamma([X])\,\gamma([Y])\quad.
\end{equation}
Now, $\gamma(\Tbb)=t$: indeed $[\Pbb^1]=2+t$, and $k^*$
is the complement of two distinct points in $\Pbb^1$. 
Next, the $\csm$ class of a join was computed in Theorem~3.13 in~\cite{MR2782886}:
\begin{equation}\label{joinf}
\csm(\one_{J(X,Y)})=((f(H)+H^m)(g(H)+H^n)-H^{m+n})\cap [\Pbb^{m+n-1}]\quad,
\end{equation}
where $\csm(\one_X)=f(H)\cap [\Pbb^{m-1}]$, $\csm(\one_Y)=g(H)\cap [\Pbb^{n-1}]$,
and $H$ denotes the hyperplane class throughout. 
Using Proposition~\ref{cstCSM}, \eqref{joinf} implies a statement on polynomial Chern
classes, which translates into
\[
t\, \gamma([J(X,Y)]) = (t\,\gamma([X])+1) (t\, \gamma([X])+1) -1\quad,
\]
as the reader may verify.
As $[J(X,Y)^\circ]=[J(X,Y)]-[X]-[Y]$, this is immediately seen to imply \eqref{homom},
concluding the proof.
\end{proof}

\begin{example}\label{P1xP1w}
By Proposition~\ref{mult2},
\begin{equation}\label{gammadot}
\gamma([\Pbb^1]\cdot [\Pbb^1]) = (2+t)^2 = 4+4t+t^2\quad.
\end{equation}
This equals the expression obtained for $\sigma(\gamma([\Pbb^1]\star [\Pbb^1]))$
in Example~\ref{P1xP1}, but reflects a very different geometric situation:
$[\Pbb^1]\star [\Pbb^1]$ is the class of a nonsingular quadric in $\Pbb^3$;
the class $[\Pbb^1]\cdot [\Pbb^1]$ does not appear to be the class of an irreducible
variety, but we can
realize $\Tbb\cdot [\Pbb^1]\cdot [\Pbb^1]$ as the class of the `open join'
$J(\Pbb^1,\Pbb^1)^\circ$ in $\Pbb^3$. Since $J(\Pbb^1,\Pbb^1)=\Pbb^3$, this
gives
\[
t\, \gamma([\Pbb^1]\cdot [\Pbb^1]) = \gamma([\Pbb^3]-2[\Pbb^1])
=(4+6t+4t^2+t^3)-2(2+t)=t(4+4t+t^2)\quad,
\]
confirming \eqref{gammadot}.
\qede\end{example}

\begin{remark}
If $\iota_X: X\to \Pbb^{m-1}$ is a closed embedding, let $\hat X\subseteq \Abb^m$
be the corresponding affine cone, and denote by $\iota_{\hat X}: \hat X \to \Pbb^m$ 
the embedding of this cone in the projectivization of $\Abb^m$. The polynomial
$\gamma([\hat X])=c_*^{\iota_{\hat X}}(\one_{\hat X})$ agrees with the polynomial 
denoted $G_{\hat X}$ in~\cite{MR2782886}. Proposition~\ref{mult2} is then a mild
generalization of Theorem~3.6 in~\cite{MR2782886}, which is the key step in the 
definition of `polynomial Feynman rules'.
\qede\end{remark}

\begin{remark}
Using the involution $\cI$, Propositions~\ref{mult} and~\ref{mult2} yield expressions
for the Euler characteristics of general linear sections of products and joins of
embedded varieties $X$, $Y$ in terms of the same information for $X$ and $Y$.
For example, as we have seen above we have
\[
\gamma_{J(X,Y)}(t)=t\, \gamma_X(t)\, \gamma_Y(t) + \gamma_X(t)+\gamma_Y(t)
\quad;
\]
applying $\cI$ gives
\begin{align*}
\chi_{J(X,Y)}(t) &= -t\, \gamma_X(-t-1)\, \gamma_Y(-t-1) + \chi_X(t)+\chi_Y(t) \\
&= \frac 1t \left((t+1) \chi_X(t)-\chi(X)\right)\left((t+1) \chi_Y(t)-\chi(Y)\right)
+\chi_X(t) + \chi_Y(t)
\quad,
\end{align*}
and reading off the coefficient of $t^\ell$ we get
\[
\chi_\ell^{J(X,Y)} = \sum_{j+k=\ell-1} (\chi_j^X-\chi_{j+1}^X)(\chi_k^Y-\chi_{k+1}^Y)
+\chi_\ell^X+\chi_\ell^Y
\]
for all $\ell>0$.
These expressions interpolate between the $\ell=1$ 
case, stating that the Euler characteristic of a general hyperplane section of
$J(X,Y)$ equals
\[
(\chi_0^X-\chi_1^X)(\chi_0^Y-\chi_1^Y)+\chi_1^X+\chi_1^Y
\]
(which is straightforward) and the $\ell=\dim X+\dim Y+1$ case, stating that 
the degree of $J(X,Y)$ is the product of the degrees of $X$ and $Y$
(\cite{MR1182558}, Example 18.7). 

Similarly, for the product $X\times Y$ embedded via the Segre embedding, the 
corresponding expressions interpolate between the well-known formulas $\chi(X\times Y)
=\chi(X)\chi(Y)$ and $\deg(X\times Y) = \binom{\dim X+\dim Y}{\dim X} (\deg X)(\deg Y)$.
The Euler characteristic of a general hyperplane section is
\[
\sum_{j\ge 1}\left(\chi_j^X(\chi_{j-1}^Y-\chi_j^Y)+\chi_j^Y(\chi_{j-1}^X-\chi_j^X)\right)
\]
and formulas for general linear sections of higher codimension are progressively
more complicated. The relative complexity of these formulas hides the simplicity of their
source, that is the homomorphism statement in Theorem~\ref{mult}.
\qede\end{remark}



\end{document}